\documentclass{article}
\usepackage{subfigure}
\usepackage{authblk}
\usepackage{cellspace}%
\usepackage{makecell}
\setcellgapes{3pt}

\usepackage{fullpage,enumerate}

\usepackage{tikz}

\usetikzlibrary{matrix} 
\usetikzlibrary{arrows,automata}
\usetikzlibrary{positioning}

\usepackage{amssymb,mathrsfs,amsmath,amsthm,bm,diagbox,float}
\usepackage{multirow}
\usetikzlibrary{matrix,arrows}

\usepackage{hyperref}

\usepackage{mathtools}

\renewcommand{\S}{\mathcal{S}}

\newtheorem{theorem}{Theorem}[section]

\newtheorem{proposition}[theorem]{Proposition}

\newtheorem{lemma}[theorem]{Lemma}
\newtheorem{conjecture}[theorem]{Conjecture}
\theoremstyle{definition}

\newtheorem{question}[theorem]{Question}

\title{Pattern-restricted permutations of small order}
\author[1]{Kassie Archer}
\author[1]{Robert P. Laudone}
\affil[1]{{\small Department of Mathematics, United States Naval Academy, Annapolis, MD, 21402}}
\affil[ ]{{\small Email: karcher@usna.edu, laudone@usna.edu}}

\date{}
\begin{document}

\maketitle

\begin{abstract}
In this paper, we enumerate 132-avoiding permutations of order 3 in terms of the Catalan and Motzkin generating functions, answering a question of B\'{o}na and Smith from 2019. We also enumerate 231-avoiding permutations that are composed only of 3-cycles, 2-cycles, and fixed points. 
\end{abstract}

\section{Introduction}

We say a permutation $\pi=\pi_1\pi_2\ldots\pi_n\in\S_n$, written in its one-line notation, contains a pattern $\sigma=\sigma_1\sigma_2\ldots\sigma_k\in\S_k$ if there is a subsequence of $\pi$ that appears in the same relative order as $\sigma$. For example, the permutation $\pi=2371645$ contains the pattern 132 since 364 is a subsequence of $\pi$ that is in the same relative order as 132. If a permutation $\pi$ does not contain a pattern $\sigma$, then we say $\pi$ avoids $\sigma.$ For example, the permutation $\pi=24152673$ avoids the pattern $321$ since no subsequence of $\pi$ of size 3 appears in decreasing order. 

Investigating the cycle type (and more generally, any algebraic properties) of pattern-avoiding permutations has been the topic of many recent papers, but proves difficult in general. There are many results about pattern-avoiding involutions (see \cite{DRS07,GM02,SS85}), which are composed of only 2-cycles and fixed points. This is primarily due to a very nice geometric symmetry involutions have when drawing the diagram of a permutation, as well as their nice behavior under RSK (in the case of $\sigma=123$). For permutations avoiding pairs of permutations, results are known for permutations composed of a single $n$-cycle (with the exception of a single pair) \cite{BC19}. There are few sets of patterns for which a complete understanding of the cycle decomposition of permutations avoiding that set is understood; one such example is the set $\{213,312\}$ \cite{G01,T01}. 

In \cite{BS19}, B\'{o}na and Smith consider permutations with the property that both the permutation and its square avoid a given pattern. They end the paper asking the following question: How many permutations avoiding 132 are of order 3 (i.e., that are composed only of 3-cycles and fixed points)? Since that paper, some related results have been found. The enumeration of permutations composed only of 3-cycles and avoiding a single pattern are found in \cite{AG21}, while permutations avoiding 231 composed only of 3-cycles and fixed points can be found in \cite{BD19}.

In this paper, we answer the question B\'{o}na and Smith posed, finding a generating function in terms of the generating functions for the Motzkin and Catalan numbers (see Theorem~\ref{thm:132-main}). 
In particular, we find that the generating function for order 3 permutations avoiding 132 is given by
\[
    A^{\{1,3\}}_{132}(z) = \dfrac{c(z^3)}{\sqrt{c(z^3)(4-3c(z^3))} - zc(z^3)},
    \]
    where $c(x)$ is the generating function for the Catalan numbers.

We additionally find a generating function enumerating permutations avoiding 231 composed only of cycles of size 1, 2, or 3 (see Table~\ref{tab:231} and Theorem~\ref{thm:231-main}), and refine this with respect to the number of cycles of each size.

\section{Background and notation}

Let $[n]=\{1,2,3,\ldots, n\}$ and let $\S_n$ denote the symmetric group on $[n]$. We write a permutation $\pi\in\S_n$ in its one-line notation as $\pi=\pi_1\pi_2\ldots\pi_n$ with $\pi_i=\pi(i).$ We denote the set of permutations in $\S_n$ that avoid the pattern $\sigma\in\S_k$ as $\S_n(\sigma).$

Given a permutation $\pi\in\S_n$, we can also consider its cycle form, which is a decomposition of the permutation into disjoint cycles. For example, $\pi =529614738$ has cycle form $\pi=(1,5)(2)(3,9,8)(4,6)(7).$ We will use the notation $c_k(\pi)$ to denote the number of $k$-cycles in the cycle form of $\pi.$ For example, if $\pi=(1,5)(2)(3,9,8)(4,6)(7),$ then $c_1(\pi) = 2$, $c_2(\pi) = 2$, $c_3(\pi) = 1$, and $c_k(\pi) =0$ for $k\geq 4.$ Furthermore, we will denote the set of permutations composed only of permutations with cycle lengths in set $S$ as $\S^S_n$. For example, the set of involutions (i.e., those permutations that are their own algebraic inverse) are composed of only 2-cycles and fixed points, so this set of permutations would be denoted by $\S_n^{\{1,2\}}.$ Similarly, we would denote by $\S_n^S(\sigma)$
the set of permutations in $\S_n^S$ that avoid the pattern $\sigma.$ We will let $a_n^S(\sigma):=|\S_n^S(\sigma)|$ and let $A^S_\sigma(z)$ be the generating function for these numbers, so that 
\[
A^S_\sigma(z) = \sum_{n\geq0} a_n^S(\sigma) z^n.
\]

We define an \emph{arc diagram} for a permutation to be the collection of the numbers in $[n]$ with arcs $a\to b$ for $a,b\in[n]$ if $a$ and $b$ are in the same cycle and there is no $c$ in the cycle with $a<c<b$. For example, for the permutation $\pi=(1,5)(2)(3,9,8)(4,6)(7)=529614738,$ the arc diagram is given by the following diagram.
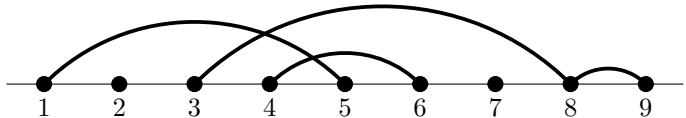
\begin{center}
    \begin{tikzpicture}
        \draw(-0.5,0) -- ++ (9,0);
    \foreach \x in {0,...,8}{
    \pgfmathtruncatemacro{\jn}{\x+1}
       \draw[circle,fill] (\x,0)circle[radius=1mm]node[below]{{\jn}};
    }
       \draw[line width=.5mm, black](2,0) to[bend left=45] (7,0);
       \draw[line width=.5mm, black](7,0) to[bend left=45] (8,0);
       \draw[line width=.5mm,black](0,0) to[bend left=45] (4,0);
       \draw[line width=.5mm,black](3,0) to[bend left=45] (5,0);
    \end{tikzpicture}
    \end{center}
Notice that the permutation $\pi=(1,5)(2)(3,8,9)(4,6)(7)$ has the same arc diagram as the one pictured above,  $\pi=(1,5)(2)(3,9,8)(4,6)(7)$. This is because there are two types of 3-cycles that are composed of the elements in $\{3,8,9\}.$ Also notice that if $\pi$ is an involution, the arc diagram will uniquely determine the permutation (since there is only one type of fixed point or 2-cycle). When working throughout the paper, we will often leave the arc diagram unlabeled, as they will often represent sub-sequences or patterns of a larger permutation. 
For example, the arc diagram pictured here:
\begin{center}
    \begin{tikzpicture}
        \draw(-0.5,0) -- ++ (3,0);
    \foreach \x in {0,...,2}{
       \draw[circle,fill] (\x,0)circle[radius=1mm]node[below]{};
    }
       \draw[line width=.5mm, black](1,0) to[bend left=45] (2,0);
    \end{tikzpicture}
    \end{center}
represents the permutation $\sigma = (1)(2,3) = 132$, which corresponds to the subsequence $264$ of the permutation $\pi=(1,5)(2)(3,9,8)(4,6)(7)=529614738$ above.

Finally, we mention that there are four symmetries of a permutation that preserve its cycle type: the identity, the inverse, reverse-complement, and reverse-complement-inverse. For this reason, when considering questions about the cycle type or algebraic order of permutations avoiding a pattern of length 3, we need only consider four cases: $\sigma\in\{123,132,231, 321\}.$ This is because 213 is the reverse-complement of 132 and 312 is the inverse of 231. 
In Section~\ref{sec:231}, we consider permutations that avoid 231, but all the results also hold for 312-avoiding permutations. Similarly, in Section~\ref{sec:132}, we consider permutations that avoid 132, but the results also hold for permutations avoiding 213. 

\section{Avoiding 231 with cycle lengths 1, 2, or 3}\label{sec:231}

In this section, we consider permutations avoiding the pattern 231 composed only of fixed points, 2-cycles, and 3-cycles. To do this, we first determine which configurations of fixed points, 2-cycles, and 3-cycles are forbidden and which are allowed. This culminates in Theorem~\ref{thm:231-main} below. 
\begin{theorem}\label{thm:231-main}
If we let \[B(t,x,y) 
= \sum_{n\geq 0} |\{\pi\in\S_n^{[3]}(231) : c_1(\pi) =i, c_2(\pi) = j, c_3(\pi) = k\}| t^ix^jy^k,\] then 
\[
B(t,x,y) = \dfrac{(1-x)^2(1-2y)}{1 - 3 x - 3 y - t + 2 x^2 + 5 x y  + t x - 5 x^2 y  - 4 t x y  -t^2y}.
\]
\end{theorem}
We can then get the generating functions $A^S_{231}(z)$ for each $S\subseteq \{1,2,3\}$ from this generating function. For example, for $S=\{2,3\}$, we have $A^{\{2,3\}}_{231}(z) = B(0,z^2,z^3).$ A summary can be found in the Table~\ref{tab:231}.

\begin{table}[ht]
\centering\makegapedcells
\renewcommand{\arraystretch}{1.6}
\begin{tabular}{|c|c|c|}
\hline
set $S$  & g.f. $A^S_{231}(z)$ & $a_n^S(231)$ for $n\in [12]$ \\ \hline
  $\{1\}$  & $\dfrac{1}{1-z}$ & $1,1,1,1,1,1,1,1,1,1,1,1$ \\ \hline
  $\{2\}$ & $\dfrac{1-z^2}{1-2z^2}$ & $0,1,0,2,0,4,0,8,0,16,0,32$\\  \hline
   $\{3\}$ & $\dfrac{1-2z^3}{1-3z^3}$ & $0,0,1,0,0,3,0,0,9,0,0,27$ \\  \hline
 $\{1,2\}$ &  $\dfrac{1-z}{1-2z}$  & $1,2,4,8,16,32,64,128,256,512,1024,2048$ \\ \hline
    $\{1,3\}$ & $\dfrac{1-2z^3}{1-z-3z^3-z^5}$ & $1,1,2,5,9,16,32,61,114,219,418,792$ \\ \hline
   $\{2,3\}$ & $\dfrac{(1-z^2)^2(1-2z^3)}{1-3z^2-3z^3+2z^4+5z^5-5z^7}$ & $0,1,1,2,5,7,17,27,57,98,193,351$\\ \hline
   $\{1,2,3\}$ & $\dfrac{(1-z)^2(1-2z^3)}{1-3z+2z^2-3z^3+6z^4-5z^5}$ & $1,2,5,12,29,71,171,411,990,2380,5722,13765$ \\ \hline
\end{tabular}
\caption{Listed here are the generating functions $A^S_{231}(x)$ for all nonempty subsets $S$ of $\{1,2,3\}$. These follow from Theorem~\ref{thm:231-main}}
\label{tab:231}
\end{table}


We will prove this theorem through a series of lemmas that dictate how different length cycles in a 231-avoiding permutation can interact with each other. Let us first consider how 2-cycles can interact with other 2-cycles in a 231-avoiding permutation. 

\begin{lemma}
Given $\pi\in\S_n(231)$, any pair of 2-cycles cannot cross. That is, if $(a,b)$ and $(c,d)$ are 2-cycles in the disjoint cycle decomposition of $\pi$ with $a<b, c<d$, and $a<c$, then either $b<c$ or $b>d.$
\end{lemma}

\begin{proof}
    Suppose instead that $(a,b)$ and $(c,d)$ are 2-cycles in the disjoint cycle decomposition of $\pi$ with $a<b, c<d$, $a<c$, and $c<b<d$. Then $bdac$ is a subsequence of $\pi$ which contains the 231 pattern $bdc.$
 \end{proof}

 This lemma says that the following configurations of 2-cycles are allowed:
  
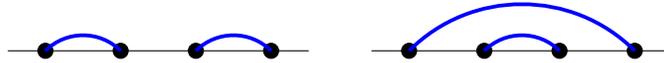
\begin{center}
    \begin{tikzpicture}
        \draw(-0.5,0) -- ++ (4,0);
    \foreach \x in {0,...,3}{
       \draw[circle,fill] (\x,0)circle[radius=1mm]node[below]{};
    }
       \draw[line width=.5mm, blue](0,0) to[bend left=45] (1,0);
       \draw[line width=.5mm, blue](2,0) to[bend left=45] (3,0);
    \end{tikzpicture} \quad \quad
        \begin{tikzpicture}
        \draw(-0.5,0) -- ++ (4,0);
    \foreach \x in {0,...,3}{
       \draw[circle,fill] (\x,0)circle[radius=1mm]node[below]{};
    }
       \draw[line width=.5mm, blue](0,0) to[bend left=45] (3,0);
       \draw[line width=.5mm, blue](1,0) to[bend left=45] (2,0);
    \end{tikzpicture} 
 \end{center}
 while the configuration below is forbidden:
 \begin{center}
    \begin{tikzpicture}
        \draw(-0.5,0) -- ++ (4,0);
    \foreach \x in {0,...,3}{
       \draw[circle,fill] (\x,0)circle[radius=1mm]node[below]{};
    }
       \draw[line width=.5mm, blue](0,0) to[bend left=45] (2,0);
       \draw[line width=.5mm, blue](1,0) to[bend left=45] (3,0);
    \end{tikzpicture} 
 \end{center}

We call the second allowed configuration a {\bf nest} of $2$-cycles. If this nest of 2-cycles are of the form $\{(e_1, f_1), (e_2, f_2), \ldots, (e_k, f_k)\}$, then we must have $e_1<e_2<\ldots<e_k<f_k<f_{k-1}<\ldots <f_1$, and we call this nest $(E,F)$ where $E=\{e_1, e_2, \ldots, e_k\}$  and $F=\{f_1, f_2,\ldots, f_k\}$. For example, the nest $(E,F) = (\{2,3,4\}, \{5,6,7\})$ corresponds to the set of two cycles $\{(2,7), (3,6), (4,5)\}$.

 Now let us consider 3-cycles. Given three elements $a<b<c$, there are two possible ways they can form a 3-cycle: $(a,b,c)$ or $(a,c,b)$. However, if the 3-cycle $(a,b,c)$ is in the cycle decomposition of $\pi$, then the sequence $bca$ appears in $\pi$, which is a 231 pattern. Therefore every 3-cycle must be of the form $(a,c,b)$ and notice that in $\pi$, $cab$ forms a 312 pattern.

 \begin{lemma}\label{lem:crossing 3 cycles}
     Given $\pi\in\S_n(231)$, there are only three possible configurations for any pair of 3-cycles. In particular, if $(a,c,b)$ and $(d,f,e)$ are 3-cycles of the form $(1,3,2)$ in the disjoint cycle decomposition of $\pi$ with $a<d$. Then one of the following must be true:
     \begin{itemize}
         \item $c<d$ (which implies all of $a,b,c$ are less than each of $d,e,f$)
         \item $c>f$ and $d<b<e$, or 
         \item $c>f$ and $e<b<f$.
     \end{itemize}
 \end{lemma}

 \begin{proof}
    There are $\frac{1}{2}\binom{6}{3} =10$ possible configurations possible when two of these 3-cycles interact. It is straightforward to check that the only three possible configurations that avoid 231 are the ones listed in the theorem. 
 \end{proof}

 This means that the only three possible configurations of two 3-cycles are those pictured below.
   
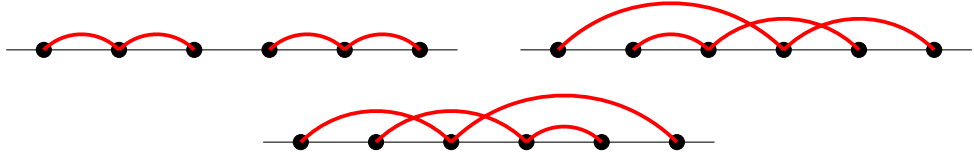
\begin{center}
    \begin{tikzpicture}
        \draw(-0.5,0) -- ++ (6,0);
    \foreach \x in {0,...,5}{
       \draw[circle,fill] (\x,0)circle[radius=1mm]node[below]{};
    }
       \draw[line width=.5mm, red](0,0) to[bend left=45] (1,0);
       \draw[line width=.5mm, red](1,0) to[bend left=45] (2,0);
       \draw[line width=.5mm,red](3,0) to[bend left=45] (4,0);
       \draw[line width=.5mm,red](4,0) to[bend left=45] (5,0);
    \end{tikzpicture} \quad\quad
         \begin{tikzpicture}
        \draw(-0.5,0) -- ++ (6,0);
    \foreach \x in {0,...,5}{
       \draw[circle,fill] (\x,0)circle[radius=1mm]node[below]{};
    }
       \draw[line width=.5mm, red](0,0) to[bend left=45] (3,0);
       \draw[line width=.5mm, red](3,0) to[bend left=45] (5,0);
       \draw[line width=.5mm,red](1,0) to[bend left=45] (2,0);
       \draw[line width=.5mm,red](2,0) to[bend left=45] (4,0);
    \end{tikzpicture}  \quad\quad
         \begin{tikzpicture}
        \draw(-0.5,0) -- ++ (6,0);
    \foreach \x in {0,...,5}{
       \draw[circle,fill] (\x,0)circle[radius=1mm]node[below]{};
    }
       \draw[line width=.5mm, red](0,0) to[bend left=45] (2,0);
       \draw[line width=.5mm, red](2,0) to[bend left=45] (5,0);
       \draw[line width=.5mm,red](1,0) to[bend left=45] (3,0);
       \draw[line width=.5mm,red](3,0) to[bend left=45] (4,0);
    \end{tikzpicture} 
 \end{center}

 As a result of this, if we have a collection of $m$ 3-cycles 
 $\{(a_i,c_i,b_i)\}$ that all cross each other, then $a_i<b_j<c_k$ for all $i$, $j$, and $k$. Furthermore there are $2^{m-1}$ such collections of $m$ crossing 3-cycles. This is clear since there are exactly two ways to add a new 3-cycle to such a collection.
 In such a collection of crossing 3-cycles, we refer to the $\{a_i\} = A$ as the first block, $\{b_i\} = B$ as the second block and $\{c_i\} = C$ as the third block.
 
 As we will see in the next lemma, when we have a 231-avoiding permutation composed only of 2-cycles and 3-cycles, any collection of crossing 3-cycles with must have that the corresponding blocks $A,B$ and $C$ are composed of sets of consecutive numbers. In the following theorem, for two sets $S$ and $T$, we write $S<T$ to denote that for each $s\in S$ and $t \in T$, $s<t$. We also write $S<t$ for a set $S$ and element $t$ if for all $s\in S$, $s<t$.
 
\begin{lemma} \label{lem:23Cycles-Crossing}
    Given $\pi \in \S_n(231)$, there are only 5 possible configurations involving a collection of crossing 3-cycles and a 2-cycle. In particular, if $A,B,$ and $C$ are the blocks of the crossing 3-cycles with $A<B<C$ and $(e,f)$ is a 2-cycle with $e<f$, one of the following must be true:
    \begin{itemize}
        \item $f < A$
        \item $A < e<f < B$
        \item $A < e < B < f < C$
        \item $B < e< f < C$
        \item $C < e$
    \end{itemize}
\end{lemma}
\begin{proof}
    It is a straightforward exercise to check this theorem for a single 3-cycle and a single 2-cycle, which in turn implies it must hold for larger collections of crossing 3-cycles. 
\end{proof}

This means that the only possible configurations containing a 2-cycle and a collection of crossing 3-cycles look like those drawn below where each blue edge corresponds to a nest of 2-cycles and the red bold edges correspond to a collection of crossing 3-cycles.
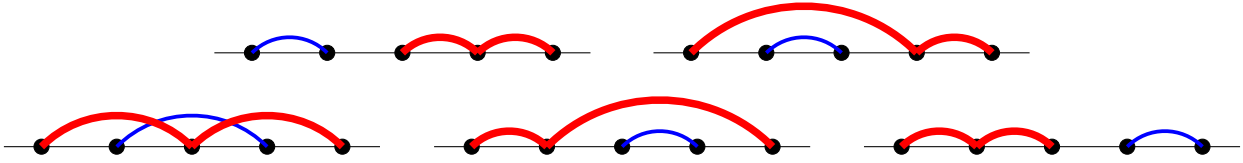
\begin{center}
    \begin{tikzpicture}
        \draw(-0.5,0) -- ++ (5,0);
    \foreach \x in {0,...,4}{
       \draw[circle,fill] (\x,0)circle[radius=1mm]node[below]{};
    }
       \draw[line width=.5mm, blue](0,0) to[bend left=45] (1,0);
       \draw[line width=1mm,red](2,0) to[bend left=45] (3,0);
       \draw[line width=1mm,red](3,0) to[bend left=45] (4,0);
    \end{tikzpicture} \quad\quad
           \begin{tikzpicture}
        \draw(-0.5,0) -- ++ (5,0);
    \foreach \x in {0,...,4}{
       \draw[circle,fill] (\x,0)circle[radius=1mm]node[below]{};
    }
       \draw[line width=.5mm, blue](1,0) to[bend left=45] (2,0);
       \draw[line width=1mm,red](0,0) to[bend left=45] (3,0);
       \draw[line width=1mm,red](3,0) to[bend left=45] (4,0);
    \end{tikzpicture}\quad\quad
        \begin{tikzpicture}
        \draw(-0.5,0) -- ++ (5,0);
    \foreach \x in {0,...,4}{
       \draw[circle,fill] (\x,0)circle[radius=1mm]node[below]{};
    }
       \draw[line width=.5mm, blue](1,0) to[bend left=45] (3,0);
       \draw[line width=1mm,red](0,0) to[bend left=45] (2,0);
       \draw[line width=1mm,red](2,0) to[bend left=45] (4,0);
    \end{tikzpicture}\quad\quad
        \begin{tikzpicture}
        \draw(-0.5,0) -- ++ (5,0);
    \foreach \x in {0,...,4}{
       \draw[circle,fill] (\x,0)circle[radius=1mm]node[below]{};
    }
       \draw[line width=.5mm, blue](2,0) to[bend left=45] (3,0);
       \draw[line width=1mm,red](0,0) to[bend left=45] (1,0);
       \draw[line width=1mm,red](1,0) to[bend left=45] (4,0);
    \end{tikzpicture}\quad\quad
        \begin{tikzpicture}
        \draw(-0.5,0) -- ++ (5,0);
    \foreach \x in {0,...,4}{
       \draw[circle,fill] (\x,0)circle[radius=1mm]node[below]{};
    }
       \draw[line width=.5mm, blue](3,0) to[bend left=45] (4,0);
       \draw[line width=1mm,red](0,0) to[bend left=45] (1,0);
       \draw[line width=1mm,red](1,0) to[bend left=45] (2,0);
    \end{tikzpicture}
    \end{center}
    
    We have seen that when $\pi$ is composed of 2-cycles and 3-cycles, the blocks  $A$, $B$, and $C$ corresponding to any collection of crossing $3$-cycles must consist of a set of consecutive numbers. This is not necessarily true for a nest $(E,F)$ of 2-cycles. For example, consider the permutation $\pi =(1,7,3)(2,6)(4,5) = 7615423$. This permutation avoids 231 and it has a nest of 2-cycles $(E,F) = (\{2,4\}, \{5,6\})$, where $E$ is not a set of consecutive numbers. In general a single nest of 2-cycles can interact with 3-cycles in any way so that each 2-cycle respects the rules given in Lemma~\ref{lem:23Cycles-Crossing}.

Finally, we consider the case were we have two non-nested 2-cycles, which will have implications for when we have two different nests of 2-cycles interacting with a collection of crossing 3-cycles.
\begin{lemma}\label{lem:forbidden}
 Given $\pi \in \S_n(231)$, if $A,B,$ and $C$ are the blocks of a collection of crossing 3-cycles with $A<B<C$ and $(e_1,f_1)$ and $(e_2,f_2)$ are non-nested 2-cycles with $e_1 < f_1<e_2<f_2$, we \textit{cannot} have the following configurations:
    \begin{itemize}
        \item $A<e_1<f_1<e_2<B<f_2<C$
        \item $A<e_1<B<f_1<e_2<f_2<C$
    \end{itemize}
\end{lemma}
\begin{proof}
    Any such configurations would contain the patterns induced by those corresponding permutations of length $7$ (containing only one 3-cycle), namely 7326145 and 7412653, both of which contain a 231 pattern.    
\end{proof}
In other words, the following configurations are forbidden: 
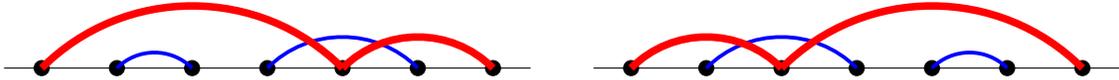
\begin{center}
    \begin{tikzpicture}
        \draw(-0.5,0) -- ++ (7,0);
    \foreach \x in {0,...,6}{
       \draw[circle,fill] (\x,0)circle[radius=1mm]node[below]{};
    }
       \draw[line width=.5mm, blue](1,0) to[bend left=45] (2,0);
       \draw[line width=.5mm, blue](3,0) to[bend left=45] (5,0);
       \draw[line width=1mm,red](0,0) to[bend left=45] (4,0);
       \draw[line width=1mm,red](4,0) to[bend left=45] (6,0);
    \end{tikzpicture} \quad\quad
        \begin{tikzpicture}
            \draw(-0.5,0) -- ++ (7,0);
    \foreach \x in {0,...,6}{
       \draw[circle,fill] (\x,0)circle[radius=1mm]node[below]{};
    }
       \draw[line width=.5mm, blue](1,0) to[bend left=45] (3,0);
       \draw[line width=.5mm, blue](4,0) to[bend left=45] (5,0);
       \draw[line width=1mm,red](0,0) to[bend left=45] (2,0);
       \draw[line width=1mm,red](2,0) to[bend left=45] (6,0);
    \end{tikzpicture} 
    \end{center}

We now have enough information to enumerate all 231-avoiding permutations composed of only 2-cycles and 3-cycles. We do this by finding a recurrence based on whether the last element $n$ is part of a 2-cycle or a 3-cycle.

\begin{lemma}\label{lem:231-S}
Suppose $T$ is the set of permutations composed of a single nonempty family of crossing $3$-cycles that also cross some positive number of 2-cycles. Then the generating function $s(x,y) = \sum_{\pi\in T} x^{c_2(\pi)}y^{c_3(\pi)}$ is given by: 
\[
s(x,y)= \dfrac{y}{1-2y}\cdot\dfrac{3x}{(1-x)^2}.
\]
\end{lemma}

\begin{proof}
Clearly, the generating function for a permutation composed of a single collection of crossing 3-cycles is given by $\dfrac{y}{1-2y}$ by Lemma~\ref{lem:crossing 3 cycles}. It only remains to determine how many ways $k$ 2-cycles can cross this collection of 3-cycles.
For a single 2-cycle, there are 3 ways the 3-cycle can cross or cross over the 2-cycles, described in Lemma~\ref{lem:23Cycles-Crossing}. Allowing two 2-cycles, following Lemmas~\ref{lem:23Cycles-Crossing} and \ref{lem:forbidden}, there are only 6 possibilities. If we only have one nest, the five possibilities are given by allowing the $b$ in the cycle $(a,c,b)$ to appear between any of the five places before, after, or between the four elements in the pair of nested 2-cycles. If we have two nests, there is one possibility, with one 2-cycle appearing below each arc of the 3-cycle. These are illustrated below.

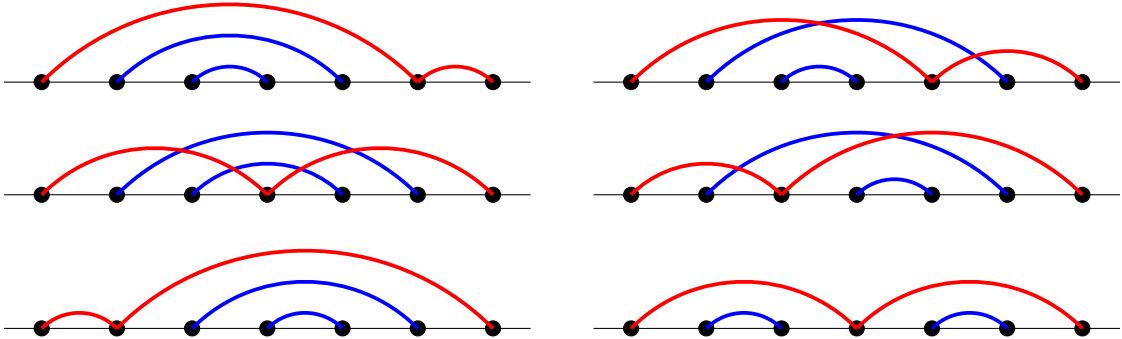
\begin{center}
    \begin{tikzpicture}
        \draw(-0.5,0) -- ++ (7,0);
    \foreach \x in {0,...,6}{
       \draw[circle,fill] (\x,0)circle[radius=1mm]node[below]{};
    }
       \draw[line width=.5mm, blue](1,0) to[bend left=45] (4,0);
       \draw[line width=.5mm, blue](2,0) to[bend left=45] (3,0);
       \draw[line width=.5mm,red](0,0) to[bend left=45] (5,0);
       \draw[line width=.5mm,red](5,0) to[bend left=45] (6,0);
    \end{tikzpicture} \quad\quad
        \begin{tikzpicture}
            \draw(-0.5,0) -- ++ (7,0);
    \foreach \x in {0,...,6}{
       \draw[circle,fill] (\x,0)circle[radius=1mm]node[below]{};
    }
       \draw[line width=.5mm, blue](1,0) to[bend left=45] (5,0);
       \draw[line width=.5mm, blue](2,0) to[bend left=45] (3,0);
       \draw[line width=.5mm,red](0,0) to[bend left=45] (4,0);
       \draw[line width=.5mm,red](4,0) to[bend left=45] (6,0);
    \end{tikzpicture} 
        \begin{tikzpicture}
        \draw(-0.5,0) -- ++ (7,0);
    \foreach \x in {0,...,6}{
       \draw[circle,fill] (\x,0)circle[radius=1mm]node[below]{};
    }
       \draw[line width=.5mm, blue](1,0) to[bend left=45] (5,0);
       \draw[line width=.5mm, blue](2,0) to[bend left=45] (4,0);
       \draw[line width=.5mm,red](0,0) to[bend left=45] (3,0);
       \draw[line width=.5mm,red](3,0) to[bend left=45] (6,0);
    \end{tikzpicture} \quad\quad
        \begin{tikzpicture}
            \draw(-0.5,0) -- ++ (7,0);
    \foreach \x in {0,...,6}{
       \draw[circle,fill] (\x,0)circle[radius=1mm]node[below]{};
    }
       \draw[line width=.5mm, blue](1,0) to[bend left=45] (5,0);
       \draw[line width=.5mm, blue](3,0) to[bend left=45] (4,0);
       \draw[line width=.5mm,red](0,0) to[bend left=45] (2,0);
       \draw[line width=.5mm,red](2,0) to[bend left=45] (6,0);
    \end{tikzpicture} 
            \begin{tikzpicture}
        \draw(-0.5,0) -- ++ (7,0);
    \foreach \x in {0,...,6}{
       \draw[circle,fill] (\x,0)circle[radius=1mm]node[below]{};
    }
       \draw[line width=.5mm, blue](2,0) to[bend left=45] (5,0);
       \draw[line width=.5mm, blue](3,0) to[bend left=45] (4,0);
       \draw[line width=.5mm,red](0,0) to[bend left=45] (1,0);
       \draw[line width=.5mm,red](1,0) to[bend left=45] (6,0);
    \end{tikzpicture} \quad\quad
        \begin{tikzpicture}
            \draw(-0.5,0) -- ++ (7,0);
    \foreach \x in {0,...,6}{
       \draw[circle,fill] (\x,0)circle[radius=1mm]node[below]{};
    }
       \draw[line width=.5mm, blue](1,0) to[bend left=45] (2,0);
       \draw[line width=.5mm, blue](4,0) to[bend left=45] (5,0);
       \draw[line width=.5mm,red](0,0) to[bend left=45] (3,0);
       \draw[line width=.5mm,red](3,0) to[bend left=45] (6,0);
    \end{tikzpicture} 
    \end{center}

In the case that we have that the 3-cycle crosses or crosses over $k$ 2-cycles, we have $3k$ possibilities: $2k+1$ that involve a single nest of 2-cycles, and $k-1$ that involve two nests of 3-cycles, similar to what we described above.  The result follows. 
\end{proof}

Now let us consider fixed points. We first see that if $\pi$ is composed of a collection of crossing 3-cycles with blocks $A, B,$ and $C$ and some collection of fixed points $G = \{g_1< g_2<\ldots< g_\ell\}$, the blocks $A,B,$ and $C$ must be sets of consecutive numbers, and we must have that no more than 1 fixed point lies below each arc of a 3-cycle. Below, we write $G_{[a,b]}$ to denote the set $\{g_a. g_{a+1}, \ldots, g_b\}$.

\begin{lemma}\label{lem:231-Q}
 Given $\pi \in \S_n(231)$, if $A,B,$ and $C$ are the blocks of a collection of crossing 3-cycles with $A<B<C$, and $G = \{g_1<g_2<\ldots<g_\ell\}$ is a collection of fixed points (i.e.~$\pi_{g_i} = g_i$ for each $i$), then we must have one of the following:
 \begin{itemize}
 \item for some $1\leq k \leq \ell+1$, $G_{[1,k-1]}<A$ and $G_{[k,\ell]}>C$, 
 \item for some $0\leq k \leq \ell$, $G_{[1,k-1]}<A$, $A<g_k<B$, and $G_{[k+1,\ell]}>C$, 
 \item for some $0\leq k \leq \ell$, $G_{[1,k-1]}<A$, $B<g_k<C$, and $G_{[k+1,\ell]}>C$, or 
 \item for some $0\leq k \leq \ell-1$, $G_{[1,k-1]}<A$, $A<g_k<B$, $B<g_{k+1}<C$, and $G_{[k+2,\ell]}>C$. 
 \end{itemize}
  Thus, if we let $Q$ be the set of permutations composed of a single nonempty family of crossing $3$-cycles that may cross over some fixed point(s). Then the generating function $q(t,y) = \sum_{\pi\in Q} t^{c_1(\pi)}y^{c_3(\pi)}$ is given by:  \[q(t,y) = \frac{y(1+2t+t^2)}{1-2y}.\]
\end{lemma}
In other words, we can have as many fixed points as we like before or after the collection of crossing 3-cycles, but we can have at most one fixed point under each arc of each 3-cycle. Furthermore, each block $A$, $B$, and $C$ will be a set of consecutive integers, as stated above. 

\begin{proof}
First, let us see that $A, B$, and $C$ must be a block of consecutive integers. If they were not, $\pi$ would contain one of the following as a pattern: 
\begin{align*}
(1,7,4)(2)(3,6,5) &= 7261554 \\ 
(1,7,5)(2)(3,6,4) &= 7263145 \\ 
(1,7,3)(2,6,5)(4) &= 7614253 \\ 
(1,7,5)(2,6,3)(4) &= 7624135 \\ 
(1,7,3)(2,5,4)(6) &= 7512463 \\ 
(1,7,4)(2,5,3)(6) &= 7521364 \\ 
\end{align*}
all of which contain a 231 pattern. Now, let us see that you must not have more than one fixed point under each arc of a 3-cycle. If $\pi$ had such a configuration, it would contain either $(1,5,4)(2)(3) = 52314$ or $(1,5,2)(3)(4)= 51342$ as a pattern, both of which contain a 231 pattern. 
\end{proof}

The last thing we need to consider is how 2-cycles and fixed points can appear in a 231-avoiding permutation. It is straightforward to see that the forbidden configurations involving 2-cycles and fixed points are the ones below:
  
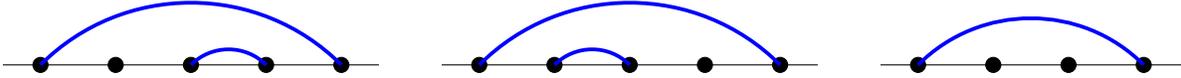
\begin{center}
    \begin{tikzpicture}
        \draw(-0.5,0) -- ++ (5,0);
    \foreach \x in {0,...,4}{
       \draw[circle,fill] (\x,0)circle[radius=1mm]node[below]{};
    }
       \draw[line width=.5mm, blue](0,0) to[bend left=45] (4,0);
       \draw[line width=.5mm, blue](2,0) to[bend left=45] (3,0);
    \end{tikzpicture} \quad \quad
        \begin{tikzpicture}
        \draw(-0.5,0) -- ++ (5,0);
    \foreach \x in {0,...,4}{
       \draw[circle,fill] (\x,0)circle[radius=1mm]node[below]{};
    }
       \draw[line width=.5mm, blue](0,0) to[bend left=45] (4,0);
       \draw[line width=.5mm, blue](1,0) to[bend left=45] (2,0);
    \end{tikzpicture} \quad\quad
    \begin{tikzpicture}
        \draw(-0.5,0) -- ++ (4,0);
    \foreach \x in {0,...,3}{
       \draw[circle,fill] (\x,0)circle[radius=1mm]node[below]{};
    }
       \draw[line width=.5mm, blue](0,0) to[bend left=45] (3,0);
    \end{tikzpicture} 
 \end{center}
since these correspond to the patterns 52431, 53241, and 4231 which all contain 231. 
The following lemma follows immediately from this.
\begin{lemma}\label{lem:231-I}
Suppose $I$ is the set of permutations composed of only 2-cycles and fixed points. Then the generating function $i(t,y) = \sum_{\pi\in I} t^{c_1(\pi)}x^{c_2(\pi)}$ is given by: 
\[
i(t,x)= \dfrac{1-x}{1-t-2x}.
\]
\end{lemma}

This gives us enough information to enumerate those permutations composed of only fixed points, 2-cycles, and 3-cycles.

\begin{lemma}\label{lem:231-R}
Suppose $R$ is the set of permutations composed of a single nonempty family of crossing $3$-cycles that also cross some positive number of 2-cycles and at least one fixed point. Then the generating function $r(t,x,y) = \sum_{\pi\in S} t^{c_1(\pi)}x^{c_2(\pi)}y^{c_3(\pi)}$ is given by: 
\[
r(t,x,y)= \dfrac{y}{1-2y}\cdot\dfrac{6xt-2x^2t + 2xt^2-x^2t^2}{(1-x)^2}.
\]
\end{lemma}

\begin{proof}
Let us first consider the case where we have two fixed points. Then each 3-cycle crosses or crosses over both fixed points and 2-cycles. Let us first note that under a single arc of the 3-cycle, we can have a nest of 2-cycles and a fixed point, but in that case, we must have that the fixed point is in the center of the nest of 2-cycles, since otherwise we would create a 231 pattern. Additionally, if a 2-cycle crosses the arcs of a 3-cycle, we cannot have a fixed point appear under the 3-cycle and completely before or after the 2-cycle. If it did, it would contain the patterns $(1,4,6)(2)(3,5)= 625134$ or $(1,3,6)(2,4)(5)=641253$, both of which contain a 231 pattern. Thus if there were two fixed points, we could not have the 2-cycles cross each 3-cycle. Instead, an arc of the 3-cycle would have to completely cross over the 2-cycle. 

Thus if there are two fixed points under the arcs of the 3-cycles, there are $k+1$ ways this can happen, namely having the two fixed points lie in the center of the two nests of 2-cycles in the $k-1$ configurations the contain two nests, or having one in the center of 1 nest that is completely under one of the arcs of the 3-cycle, while the other fixed point is under the under arc. 

On the other hand, if there is only one fixed point, there are $4k+2$ possible configurations. Since there are $k+1$ ways to place $k$ 2-cycles that do not cross, there are clearly $2(k+1)$ ways where no 2-cycle crosses each 3-cycle (since the fixed point can go in exactly one place under the first or second arc of the 3-cycle). There are $2k-1$ ways to place the 2-cycles so that the nest does cross the 3-cycles. In each case, there is exactly one place the fixed point can go, except in the case where $a<E<b<F<c$ for the nest $(E,F)$ and any 3-cycle $(a,c,b)$, in which case there are two places the fixed point can go, namely before or after $b$.  
Thus, we have the following:
\[r(t,x,y)= \dfrac{y}{1-2y}\cdot\left(\dfrac{4xt}{(1-x)^2}+\dfrac{2xt}{(1-x)} + \dfrac{xt^2}{(1-x)^2}+\dfrac{xt^2}{1-x}\right).\]
\end{proof}

\begin{proof}[Proof of Theorem~\ref{thm:231-main}]
Let $\pi\in\S_n^{[3]}(231).$ As an arc diagram, this permutation is a concatenation of collections of crossing 3-cycles (that may or may not cross over fixed points and 2-cycles), separated by possibly empty involutions. In other words, $\pi = \alpha_1\oplus\beta_1\oplus\alpha_2\oplus\beta_2 \oplus \cdots \oplus \beta_k\oplus\alpha_{k+1}$, where for each $i$, $\alpha_i$ is a (possibly empty) involution avoiding 231, and $\beta_i$ is a nonempty set of crossing 3-cycles that may or may not cross over fixed points and 2-cycles. The generating function $i(t,x)$ defined in  Lemma~\ref{lem:231-I} is the generating function for involutions, and as defined in Lemmas~\ref{lem:231-S}, \ref{lem:231-Q}, and \ref{lem:231-R}, the generating function for the  nonempty set of crossing 3-cycles that may or may not cross over fixed points and 2-cycles is $q(t,y) +s(x,y)+r(t,x,y)$. It follows that
\[
B(t,x,y) =  \frac{i(t,x)}{1-i(t,x)\cdot(q(t,y) +s(x,y)+r(t,x,y))},
\]
which is equivalent to the statement in the theorem.
\end{proof}

\section{132-avoiding permutations of order 3}\label{sec:132}


In this section, we enumerate the set of 132-avoiding permutations of order 3, i.e.~those composed of only 3-cycles and fixed points. 

In \cite{AG21}, the authors find a generating function for 132-avoiding permutations composed only of 3-cycles. Before incorporating fixed points, we find a slightly different (but equivalent) generating function for these permutations, that we will then be able to refine to include fixed points.

In this section, we will use the notation $[a,b,c]$ with $a<b<c$ to denote the pair of arcs $a\to b$ and $b \to c$ associated to the 3-cycle containing those three elements, and we will often refer to $[a,b,c]$ itself as a 3-cycle, a 3-arc, or just an arc. Note that since there are two forms of $3$-cycles, the arc $[a,b,c]$ could correspond to the 3-cycle $(a,b,c)$ or $(a,c,b)$.


\begin{lemma}\label{lem: firstComesFirst}  
Given $\pi \in \S_n(132)$, then any pair of $3$-cycles in the cycle decomposition of $\pi$ must appear in the arc diagram of $\pi$ as one of the following configurations. 
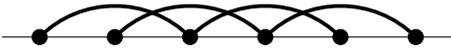
\begin{enumerate}[(A.)]
\item \strut

   \begin{tikzpicture}
        \draw(-0.5,0) -- ++ (6,0);
    \foreach \x in {0,...,5}{
       \draw[circle,fill] (\x,0)circle[radius=1mm]node[below]{};
    }
       \draw[line width=.5mm](0,0) to[bend left=45] (2,0);
       \draw[line width=.5mm](2,0) to[bend left=45] (4,0);
       \draw[line width=.5mm](1,0) to[bend left=45] (3,0);
       \draw[line width=.5mm](3,0) to[bend left=45] (5,0);
    \end{tikzpicture}  
\item\strut

    
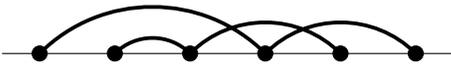
\begin{tikzpicture}
        \draw(-0.5,0) -- ++ (6,0);
    \foreach \x in {0,...,5}{
       \draw[circle,fill] (\x,0)circle[radius=1mm]node[below]{};
    }
       \draw[line width=.5mm](0,0) to[bend left=45] (3,0);
       \draw[line width=.5mm](2,0) to[bend left=45] (4,0);
       \draw[line width=.5mm](1,0) to[bend left=45] (2,0);
       \draw[line width=.5mm](3,0) to[bend left=45] (5,0);
    \end{tikzpicture}  
\item\strut

        
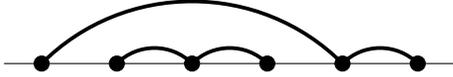
\begin{tikzpicture}
        \draw(-0.5,0) -- ++ (6,0);
    \foreach \x in {0,...,5}{
       \draw[circle,fill] (\x,0)circle[radius=1mm]node[below]{};
    }
       \draw[line width=.5mm](0,0) to[bend left=45] (4,0);
       \draw[line width=.5mm](2,0) to[bend left=45] (3,0);
       \draw[line width=.5mm](1,0) to[bend left=45] (2,0);
       \draw[line width=.5mm](4,0) to[bend left=45] (5,0);
    \end{tikzpicture}  
\end{enumerate}
Furthermore, if one 3-cycle is of the form $(1,2,3)$ and the other is of the form $(1,3,2)$, they must appear in the arc diagram as configuration (C.) above.
\end{lemma}

\begin{proof}
       There are $\frac{1}{2}\binom{6}{3} = 10$ possible configurations involving these two 3-arcs, and there are 4 possible ways to assign cycle forms to each arc, for a total of 40 possible permutations. It is a straightforward exercise to check that if the two 3-cycles are of the same form, then any configuration other than the ones listed above must contain a 132 pattern. Similarly, if the two cycles have different forms, it is straightforward to check that configuration (C.) is the only one that avoids 132. 

\end{proof}

Let's consider what Lemma~\ref{lem: firstComesFirst}  means for the arc diagram of a permutation composed only of 3-cycles. Given such a permutation on $[n]=[3m]$, and the collection of 3-cycles $[a_i,b_i,c_i]$ with $i\in[m]$ that compose $\pi$, we now know that each arc starts in one of the first $m=n/3$ positions. That is, $A = \{a_1, a_2, \ldots, a_m\} = \{1,2,\ldots, m\}.$ If we fix the sets $B=\{b_1, b_2,\ldots, b_m\}$ and $C=\{c_1,c_2,\ldots, c_m\},$ then we also know that $b_i<b_{i+1}$ and $c_i<c_{i+1}$ for all $i$. 
That is, if we know the second and third entries for each 3-arc as sets, the corresponding arcs connecting the second and third entries must connect a given $b$ to the closest unclaimed $c$.

Furthermore, take $D_{\pi}$ to be the word $w_{j}=0$ if $j=b_i-m$ for some $i$ and let $w_j=1$ if $j=c_i-m$ for some $i$, which must be a Dyck word.  Indeed, if it were not, there would be some $c_i>b_j$ for $i>j$, which is clearly not possible by Lemma~\ref{lem: firstComesFirst}.
For example, consider the permutation \[11 \ 10 \ 9 \ 7 \ 4 \ 3 \ 5 \ 2 \ 6 \ 8 \ 12 \ 1 = (1,11,12)(2,10,8)(3,9,6)(4,7,5)\] with arc diagram
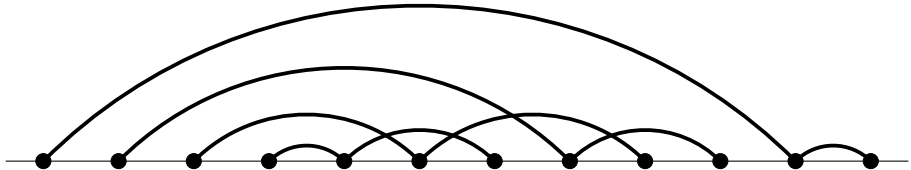
\begin{center}
        \begin{tikzpicture}
        \draw(-0.5,0) -- ++ (12,0);
    \foreach \x in {0,...,11}{
       \draw[circle,fill] (\x,0)circle[radius=1mm]node[below]{};
    }
       \draw[line width=.5mm](0,0) to[bend left=45] (10,0);
        \draw[line width=.5mm](10,0) to[bend left=45] (11,0);
       \draw[line width=.5mm](1,0) to[bend left=45] (7,0);
        \draw[line width=.5mm](7,0) to[bend left=45] (9,0);
       \draw[line width=.5mm](2,0) to[bend left=45] (5,0);
        \draw[line width=.5mm](5,0) to[bend left=45] (8,0);
       \draw[line width=.5mm](3,0) to[bend left=45] (4,0);
        \draw[line width=.5mm](4,0) to[bend left=45] (6,0);
    \end{tikzpicture}  
\end{center}
 Since $A = \{1,2,3,4\}, B=\{5,6,8,11\},$ and $C=\{7,9,10,12\}$, the associated Dyck word is $D_\pi = 00101101$.

Notice also that in the arc diagram associated to this Dyck word, we always connect the $i$-th 0 to the $i$-th 1, since $b_i$ is connected to $c_i$ in the arc diagram for the permutation. This implies that the part of the arc diagram induced by arcs between $B$ and $C$ are exactly the non-nested matchings associated to a Dyck word.

We will call a maximal block of consecutive $0$'s in a Dyck word a {\bf free block} if, in the corresponding arc diagram, they are connected to a block of consecutive $1$'s in the non-nested matching associated to the Dyck word. (We also call the corresponding block of consecutive 1's free.) For example, in the Dyck path $00101101$, all free blocks are size 1. However, in the Dyck path $00011101$ the first group of three $0$'s form a free block of size $3$ because they are connected to a block of consecutive $1$'s. 

We can associate to each permutation a composition $X=(x_1, x_2, \ldots, x_k)$ of $m$, which we will call the {\bf free composition}, based on the size of the free blocks. For example, the Dyck word $0001100110001111$ is associated to the composition $(2,1,1,1,3)$ since we can see the corresponding free blocks of 0's associated to consecutive 1's: 

\[00\ 0\ {\bf11}\ 0 \ 0\ {\bf1} \ {\bf1}\ 000\ {\bf1}\ {\bf111}.\]


Let us say that a Dyck word $D$ of length $2m$ has a {\bf hit} at position $2j$ for $j\in\{0,1,2,\ldots, m\}$ if the number of 0's and the number of 1's in $D$ are both equal to $j$. (In other words the Dyck path associated to Dyck word $D$ touches the ground.) For example, the Dyck word $00101101$ has 3 hits: at 0, 6, and 8.

\begin{lemma}\label{lem: dyck words}
    Let $\pi \in \S_{3m}(132)$ be composed only of 3-cycles $[a_i,b_i,c_i]$ for $i\in[m]$ and let $D_{\pi}$ be the associated Dyck word. Then if the $j$-th and $(j+1)$-st 1 are in a free block together, then cycles associated to $[a_j,b_j,c_j]$ and $[a_{j+1},b_{j+1},c_{j+1}]$ must be of the same form. Furthermore, if $[a_j,b_j,c_j]$ and $[a_{j+1},b_{j+1},c_{j+1}]$ are of different forms, then the Dyck word has a hit at position $2j$. 
\end{lemma}

\begin{proof}
    Recall that by Lemma~\ref{lem: firstComesFirst}, if two cycles $[a_j,b_j,c_j]$ and $[a_{j+1},b_{j+1},c_{j+1}]$ are of different forms, then we must have that all of $\{a_j,b_j,c_j\}$ are greater than $a_{j+1}$ and less than $b_{j+1}$ and $c_{j+1}$. This means that in the Dyck path $D_\pi$, the $j$-th and $(j+1)$-st 0's and 1's appear in the order 0101, and so there is no chance of them being part of a free block together. Moreover, since the $j$-th 1 appears before the $(j+1)$-st 0, we must have that the Dyck word has a hit here. 
\end{proof}

The next lemma states that just as we can break 0's and 1's (associated to $B$ and $C$) into free blocks, we can also break $A$ into blocks of the same size using the reverse of the composition $X$, and that elements in corresponding blocks must connect to each other in the arc diagram, as pictured. In this lemma, we will write $A$ and $B$ as disjoint unions of consecutive subsets determined by $X$, namely, that $A=\bigcup_{j=1}^k A_j$ and $B = \bigcup_{j=1}^k B_j$ so that for each $j$,
\begin{itemize}
\item the elements in $A_j$ are less than the elements of $A_{j+1}$,
\item the elements in $B_j$ are less than the elements of $B_{j+1}$
\item  $|A_j|=x_{k-j+1}$ and $|B_{j}|=x_j$. 
\end{itemize}
For example, if $A = \{1,2,3,4,5,6\}$, $B=\{7,8,9,12,13,14\}$, and $X= (2,1,3)$, then $A_1 = \{1,2,3\}, A_2 = \{4\}, A_3=\{5,6\}$ and $B_1=\{7,8\}, B_2=\{9\}, B_3=\{12,13,14\}$. 

\begin{lemma}\label{lem: splitOnes}
    Let $\pi \in \S_{3m}(132)$ be composed only of 3-cycles $[a_i,b_i,c_i]$ for $i\in[m]$, let $D_{\pi}$ be the associated Dyck word, and let $X = (x_1, \ldots, x_k)$ be the 
    free composition of $D_{\pi}.$ Let $A=\bigcup_{j=1}^k A_j$ and $B = \bigcup_{j=1}^k B_j$ as described above. Then the arc diagram obtained by considering only the arcs between $A$ and $B$ send elements in $A_{k-j+1}$ to elements in $B_j$.
\end{lemma}

\begin{proof}
We need only see that if two elements $b_i<b_{i+1}$ of $B$ are not in a free block together, then we must have that $a_i>a_{i+1}.$ If the cycles associated to $[a_i,b_i,c_i]$ and $[a_{i+1}, b_{i+1},c_{i+1}]$ are of different forms, then this lemma follows from Lemma~\ref{lem: firstComesFirst}. Therefore suppose the two 3-cycles are of the same form and for the sake of contradiction, suppose the lemma doesn't hold. Then we have $a_i<a_{i+1}<b_i<b_{i+1}$ and by Lemma~\ref{lem: firstComesFirst}, $c_i<c_{i+1}$. Since they are not in a free block together, we must also have that there is some $j<i$ so that either $b_i<c_j<b_{i+1}$ or some $j>i$ so that $c_i<b_j<c_{i+1}$.
In the first case, the diagram corresponding to $b_j,b_i,b_{i+1}, c_j,c_i,c_{i+1}$ looks like the following. 
    
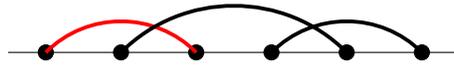
\begin{center}
    \begin{tikzpicture}
        \draw(-0.5,0) -- ++ (6,0);
    \foreach \x in {0,...,5}{
       \draw[circle,fill] (\x,0)circle[radius=1mm]node[below]{};
    }
     \draw[line width=.5mm, color=red](0,0) to[bend left=45] (2,0);
       \draw[line width=.5mm](1,0) to[bend left=45] (4,0);
       \draw[line width=.5mm](3,0) to[bend left=45] (5,0);
    \end{tikzpicture}

\end{center}
 Extending this to include $a_j,a_i,a_{i+1}$, taking $a_i<a_{i+1}$, we would have something like this:
    
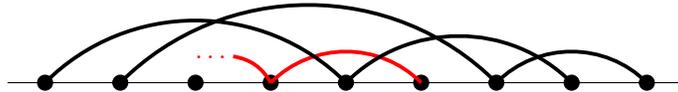
\begin{center}
    \begin{tikzpicture}
        \draw(-0.5,0) -- ++ (9,0);
    \foreach \x in {0,...,8}{
       \draw[circle,fill] (\x,0)circle[radius=1mm]node[below]{};
    }
     \draw[line width=.5mm, color=red](3,0) to[bend left=45] (5,0);
     \draw[line width=.5mm, color=red](2.5,.35) to[bend left=25] (3,0);
       \draw[line width=.5mm](4,0) to[bend left=45] (7,0);
       \draw[line width=.5mm](6,0) to[bend left=45] (8,0);
         \draw[line width=.5mm](0,0) to[bend left=45] (4,0);
           \draw[line width=.5mm](1,0) to[bend left=45] (6,0);
           \node at (2.2,.35) {\color{red}$\ldots$};
    \end{tikzpicture}
\end{center}
If the cycles are of the form $(1,2,3)$, then if $[a_j,b_j,c_j]$ is of the form $(1,2,3)$ then $\pi_{a_i}\pi_{a_{i+1}}\pi_{b_j} = b_ib_{i+1}c_j$ is a 132 pattern. If $[a_j,b_j,c_j]$ is of the form $(1,3,2)$, then $\pi_{b_j}\pi_{b_{i+1}}\pi_{c_{i+1}} = a_jc_{i+1}a_{i+1}$, so in order to avoid 132, we must have  $a_j>a_{i+1}$, but then $\pi_{a_i}\pi_{a_{i+1}}\pi_{a_j} = b_ib_{i+1}c_j$ is a 132 pattern. A similar argument holds if both of the cycles are of the form $(1,3,2)$, and a similar argument also holds in the case that there is some  $j>i$ so that $c_i<b_j<c_{i+1}$.
\end{proof}

Below, let $C_r$ denote the $r$-th Catalan number. 
\begin{lemma} \label{lem: FirstCatalan}
Suppose $D$ is a Dyck word with $\ell$ hits and with free composition $X= (x_1, x_2, \ldots, x_k)$. Then there are $2^{\ell-1} C_{x_1}C_{x_2}\cdots C_{x_r}$ permutations $\pi \in \S_{3m}(132)$ composed only of 3-cycles with $D_\pi=D$. 
\end{lemma}

\begin{proof}
Let  $\pi \in \S_{3m}(132)$ be composed only of 3-cycles $[a_i,b_i,c_i]$ with $i\in[m]$ and let $D_\pi$ be the associated Dyck word. 
    Being a free block means that the second and third entries of all the 3-cycles corresponding to this the block are each consecutive blocks of integers and Lemma \ref{lem: splitOnes} shows the same is true of your first entries. Suppose a subset $B'\subseteq B$ corresponds to a free block of size $r$ (and $A'$ and $C'$ are the corresponding blocks in $A$ and $C$ respectively). For any $\sigma \in \S_r(132)$, we claim that if you arrange your first entries according to $\sigma$, this will not create a $132$ in $\pi$.

   Since these are blocks of consecutive integers that form a free block, all associated 3-cycles are of the same form. Thus we must have that either all elements of $B$ appear before all elements of $C$ which appear before all elements of $A$ (or all elements of  $C$ appear before all elements of $A$ which appear before all elements of $B$.) In the first case, one way a 132 could appear is if the 1 appeared in $B$ and the $32$ appeared in $C$. However, by Lemma~\ref{lem: firstComesFirst}, we know that if $b_1<b_2<\ldots<b_r$ then $c_1<c_2<\ldots<c_r$ and so the elements in $C$ are in increasing order, so cannot be the 32 in a 132 pattern. The only other way a 132 pattern could appear is within $A$ or within $C$. Therefore, we make the requirement that the way $C$ is arranged avoids 132, of which there are $C_r$ ways. This automatically implies that the elements of $A$ avoid $132^{-1}=132$ as well. A similar argument works for the other case, when the blocks appear in the order $C,A,B$ in the one-line notation. 
   
   Since we can do this for each free block, there are $C_{x_1}C_{x_2}\cdots C_{x_r}$ possible arc diagrams associated to 132-avoiding permutations. It remains to consider what the cycle form for each arc is. 
   
  We know from Lemma~\ref{lem: dyck words} that if we switch cycle forms, we must have a hit in the corresponding Dyck word. We will see here that if we have a hit, we are indeed allowed to switch cycle forms. Notice that if a Dyck path hits the ground, the corresponding arc diagram looks like the following, where thickened arcs represent larger arrangements of 3-arcs.    
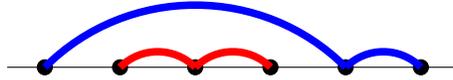
\begin{center}
    \begin{tikzpicture}
        \draw(-0.5,0) -- ++ (6,0);
    \foreach \x in {0,...,5}{
       \draw[circle,fill] (\x,0)circle[radius=1mm]node[below]{};
    }
       \draw[line width=1mm, blue](0,0) to[bend left=45] (4,0);
       \draw[line width=1mm, blue](4,0) to[bend left=45] (5,0);
       \draw[line width=1mm,red](1,0) to[bend left=45] (2,0);
       \draw[line width=1mm,red](2,0) to[bend left=45] (3,0);
    \end{tikzpicture}
    \end{center}
    To prove the lemma, suppose the thickened blue arc represents some arrangement of 3-cycles of the form $(1,3,2)$ and the thickened red arc represents some allowable arrangement of 3-cycles of the form $(1,2,3)$. 
    Both thickened arcs avoid $132$. Because all the red entries are consecutive we cannot have ${\color{red} 13}{\color{blue} 2}$. Furthermore, all the blue entries to the left of the red entries are larger than all the red entries. So we cannot have ${\color{blue} 13}{\color{red} 2}$, nor can we have ${\color{blue} 1} {\color{red} 32}$.
    Finally, we cannot have ${\color{red} 1}{\color{blue} 32}$ because all the entries to the right of the red entries that are larger than them must be increasing.
 If we swap from $(1,3,2)$ to $(1,2,3)$ when the Dyck path hits the ground, similar arguments show we cannot get a $132$.
\end{proof}

Overall, this shows that our permutations look like:
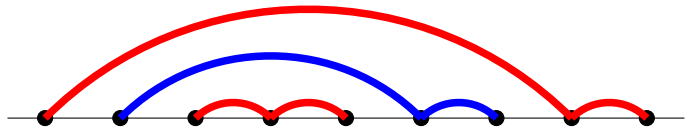
\begin{center}
    \begin{tikzpicture}
        \draw(-0.5,0) -- ++ (9,0);
    \foreach \x in {0,...,8}{
       \draw[circle,fill] (\x,0)circle[radius=1mm]node[below]{};
    }
        \draw[line width=1mm,red](0,0) to[bend left=45] (7,0);
       \draw[line width=1mm, blue](1,0) to[bend left=45] (5,0);
       \draw[line width=1mm, blue](5,0) to[bend left=45] (6,0);
       \draw[line width=1mm,red](2,0) to[bend left=45] (3,0);
       \draw[line width=1mm,red](3,0) to[bend left=45] (4,0);
       \draw[line width=1mm,red](7,0) to[bend left=45] (8,0);
    \end{tikzpicture}
\end{center}
where here, each thickened arc represents an allowable arrangement of 3-cycles of a fixed type. 

   To summarize what we have so far, we know that permutations in $S_{3m}(132)$ that are composed of $m$ 3-cycles are each associated to a Dyck word and that the number of such permutations associated to a given Dyck word depends on its free composition and the number of hits it has. 
   
     In \cite[Lemma 3.6]{AG21}, the authors show that there are $M_{k-1}$ Dyck words with free composition $(x_1,\dots,x_k)$ via a bijection where $M_{k-1}$ is the $(k-1)$-st Motkzin path. We will refine this bijection to determine the number of Dyck words with free composition $(x_1,\dots,x_k)$ that have exactly $\ell$ hits. 
     
     The bijection from \cite{AG21} works as follows: start with Dyck word $D$ of length $2m$ associated to composition $(x_1,\dots,x_k)$ of $m$. We can reduce this Dyck word to a Dyck word $D'$ of length $2k$ by taking each free block of 0's and compressing it to a single 0 and each free block of 1's and rewriting it with a single 1. Next, build a Motzkin path $M(D) = m_1m_2\ldots m_k$ by setting $m_i=u$ if both the $i$-th 0 and $i$-th 1 in $D'$ are followed by a 0, setting $m_i=d$ if both the $i$-th 0 and $i$-th 1 are followed by 1, and setting $m_i=f$ if the $i$-th 0 is followed by 1 but the $i$-th 1 is followed by 0. Because $D'$ is reduced, we will never have the $i$-th 0 followed by 0 and the $i$-th 1 followed by 1. For a fixed composition, this can clearly be undone. 
     
     For example, consider the Dyck word $D=000110011000111101$ of length 18 associated to the composition $X=(2,1,1,1,3,1)$ of length $k=6$. We can find $D' = 001001101101$ of length 12, which in turn is associated to the Motzkin path $M^D=ududf$ of length $k-1=5.$
   
   \begin{lemma} \label{lem: DyckMotzkinBij}
    Let $X = (x_1,\dots,x_k)$ be a composition of $n$. Dyck words with free composition $X$ and $\ell$ hits are in bijection with Motzkin paths of length $k-1$ with $\ell-2$ flats at level zero.
\end{lemma}
\begin{proof}
    Under the bijection described in \cite{AG21} (and above), whenever the Dyck word has a hit (except for the times at the beginning and end of the Dyck word), this corresponds to the Motzkin path having a flat at level zero. Indeed, if the Dyck word hits the ground, we must have that the reduced Dyck word $D'$ hits the ground as well. In this case, we will have a flat since the the last 0 in that initial part of the word must be followed by 1, while the last 1 in the part of the word must be followed by 0. Furthermore, this flat occurs at level zero since the initial subword itself is a Dyck word, and thus the corresponding Motzkin word to that part of the Dyck word is also itself a Motzkin path.
\end{proof}

It is well known that the generating function for Motzkin paths satisfies $m(x) = xm(x) + x^2m(x)^2 +1$ by decomposing the paths. We can similarly modify this to solve for generating function for the number $M_{k,j}$ of Motzkin paths of length $k$ with $j$ flats at level zero. Namely, if $m(t,x) = \sum_{k,j\geq 0} M_{k,j}x^kt^j$ then $m(t,k) = 1+txm(t,k) + x^2m(t,x)m(x)$. This gives us the generating function:
\[
m(t,x) = \dfrac{2}{1 -2tx +x +\sqrt{1-2x-3x^2}}.
\]
Taken altogether, we get an alternate version of \cite[Theorem 3.12]{AG21}, stated below.  Here $c(x)$
denotes the generating function for the Catalan numbers: 
 \[
 c(x) = \dfrac{1 -\sqrt{1 - 4x}}{2x}.
 \]

\begin{theorem}
Suppose $A^{\{3\}}_{132}(z)$ is the generating function for the number of 132-avoiding permutations composed only of 3-cycles, then 
\[
A^{\{3\}}_{132}(z) = 2(c(z^3)-1)m(2,c(z^3)-1).
\]
\end{theorem}

\begin{proof}
Here, we have 2 initial choices for the cycle form of the arcs and have 2 more choices (to keep the same cycle form or switch) every time the Motzkin path has a flat at level zero. Additionally, we can build a composition of the length one more than the length of a the Motzkin path, and there are Catalan-many choices for the permutation associated to each element of the Motzkin path. Notice that since our compositions are not weak compositions, we must use $c(z^3)-1$, not allowing for the empty Catalan object. 
\end{proof}
    
Finally, we want to understand where fixed points can appear relative to the 3-cycles in a 132-avoiding permutation of order 3. 

\begin{lemma} \label{lem: fixedPoints}
    Given $\pi \in \S_n(231)$ of order three, fixed points can only occur after all the first entries in positions, specifically in places where the Dyck word corresponding to the second and third entries have hits.
\end{lemma}

\begin{proof}
    First, let us see that it is impossible to place a fixed point in between any of the first entries. Suppose we have
    
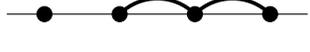
\begin{center}
    \begin{tikzpicture}
        \draw(-0.5,0) -- ++ (4,0);
    \foreach \x in {0,...,3}{
       \draw[circle,fill] (\x,0)circle[radius=1mm]node[below]{};
    }
       \draw[line width=.5mm](1,0) to[bend left=45] (2,0);
       \draw[line width=.5mm](2,0) to[bend left=45] (3,0);
    \end{tikzpicture},
\end{center}
as a sub-diagram in our permutation. If the 3-cycle is of the form $(1,2,3)$ or $(1,3,2)$ then the fixed point along with the first and third entry of the 3-cycle make a $132$.

So fixed points must occur after all of the first entries of our $3$-cycles. If we place a fixed point where the Dyck path hits the ground, it will not create a 132. Indeed, the fixed point is exactly one more than the largest element to its left that is smaller than it. So it cannot be the $3$ of a $132$.
It also cannot be the $1$ of a $132$ because all the entries to its right that are larger than it are increasing; these correspond to the $0$'s in $D_\pi$.
Finally, it cannot be the $2$ of a $132$ because any increasing pair of elements to its left are either both larger than the fixed point (first entries paired with $0$'s and $1$'s occurring after the fixed point) or both smaller (any of the other entries to the left of the fixed point).

Furthermore, we note that we can place as many fixed points as we would like into these spots where the Dyck path hits the ground because this is just introducing an increasing subsequence in these spots, but all the same arguments from Lemma \ref{lem: fixedPoints} show these fixed points do not create a $132$.
\end{proof}



We can now state the main theorem of this section. 

\begin{theorem}
    For $m\geq3$, the number of permutations $\pi\in\S_n(132)$ with $r$ fixed points and $m$ 3-cycles is
    \[
    \sum_{k=1}^m\sum_{\ell=2}^{k+1}2^{\ell-1}\binom{r+\ell}{r} M_{k-1,\ell-2}\sum_{X \in\mathcal{P}_{m,k}} C_{x_1}\cdots C_{x_k}    \]
    where $\mathcal{P}_{n,k}$ are all compositions of $m$ of length $k$ and $M_{k-1,j-2}$ is the number of Motzkin paths of length $k-1$ with $j-2$ flats at level 0. 
\end{theorem}


\begin{proof}
    We sum over all composition lengths $k$ of $m$ and possible number of hits $\ell$ that a Dyck word has. For each composition $X \in \mathcal{P}_{m,k}$, Lemma \ref{lem: DyckMotzkinBij} shows there are $M_{k-1,\ell-2}$ Dyck paths of this type that have $\ell$ hits. By Lemmas \ref{lem: FirstCatalan} each of these Dyck paths corresponds to $2^{\ell-1}C_{x_1}\cdots C_{x_n}$ elements in $\S_{3m}(132)$ composed only of 3-cycles.
    
    We can then place our $r$ fixed points anywhere the Dyck word has a hit by Lemma \ref{lem: fixedPoints}. The number of ways to distribute the $r$ fixed points into the $j+1$ positions is $\binom{r+j}{r}$, which proves the result. 
\end{proof}






Finally, let us determine the corresponding generating function for order 3 permutations avoiding 132. 
\begin{theorem}\label{thm:132-main}
    Let $F(t,x) = 1+xtM(t,x)$. The generating function for order $3$ permutations avoiding $132$ is
    \[
    A^{\{1,3\}}_{132}(z) = \dfrac{1}{1-z}\cdot F\bigg(\frac{2}{1-z}, c(z^3)-1\bigg)
    \]
    which is equivalent to 
      \[
    A^{\{1,3\}}_{132}(z) = \dfrac{c(z^3)}{\sqrt{c(z^3)(4-3c(z^3))} - zc(z^3)}.
    \]
\end{theorem}

\begin{proof}
    Every time the Motzkin path has a flat at level zero, we are allowed to place fixed points (Lemma \ref{lem: fixedPoints}) and change 3-cycle type (Lemma \ref{lem: FirstCatalan}). We use $\frac{2}{1-z}$ at each flat of the Motzkin path to account for this.

    Lemma \ref{lem: FirstCatalan} enumerates the arrangements of the first entries by products of Catalan numbers. We have to add an additional $\frac{1}{1-z}$ to account for the fixed points we can add at the end of the Motzkin path where we cannot change our 3-cycle type. It follows that
    \[
    A^{\{1,3\}}_{132}(z) = \frac{1}{1-z}+ \frac{2}{(1-z)}(c(z^3)-1)\cdot M\bigg(\frac{2}{1-z}, c(z^3)-1\bigg)
    \]
    which is equivalent to the theorem statement.
\end{proof}

It was conjectured in \cite{BS19} that the growth rate of order 3 permutations avoiding 132 would be less than or equal to 2. Using this generating function we can confirm this conjecture by noting that the growth rate of $A^{\{1,3\}}$ is $1.89$.

\section{Open Questions}

In this paper, we found $a_n^{\{1,3\}}(132)$ and $a_n^{\{1,2,3\}}(231)$. There are some very natural follow-up questions. For example, we could ask ourselves how 2-cycles may factor into the 132-avoiding permutations. 
\begin{question}
    How many 132-avoiding permutations are there that are composed of only 3-cycles, 2-cycles, and fixed points?
\end{question}
This question seems more difficult than the equivalent question for 231, not only because the way 3-cycles and fixed points interact is more complicated, needing Dyck paths and Motzkin paths to keep track of them, but also because any 2-cycles in the permutation seem to be able to interact with the 3-cycles in interesting ways. That said, this does seem like a feasible task.

Let us also consider permutations avoiding the pattern 123. 
In the case of 3-cycles and fixed points, the answer is trivial, as stated in the proposition below.  
\begin{proposition}
    $|\S^{\{1,3\}}_n(123)| = 0$ for $n\geq 9$.
\end{proposition}

\begin{proof}
    It is straightforward to check the cases for $9\leq n \leq 14,$ so let $n>14.$
    Both types of 3-cycles contain an increasing sub-sequence. For 3-cycles of the form $(1,3,2)$ this means their second and third entries must be nested. But then to avoid a $123$ from their first entries for $n \geq 3$ the first entries must be decreasing. Since the second and third entries are nested, this forces their second entries to be increasing. So we cannot have more than two 3-cycles of the form $(1,3,2)$.

    An identical argument shows we cannot have more than two 3-cycles of the form $(1,2,3)$.
    We also cannot have more than two fixed points. As a result, for $n > 14$, $|\S^{\{1,3\}}_n(123)| = 0$.
\end{proof}

In fact, any permutation avoiding 123 must have at most two 3-cycles and at most two fixed points. It is well-known that in $\S_n(123)$, there are $\binom{n}{\lfloor \frac{n}{2}\rfloor}$ involutions and for even $n$, there are $\binom{n-1}{\frac{n}{2}}$ fixed point free involutions 
\cite{DRS07,SS85}. So, let us consider the following question.

\begin{question}
    How many 123-avoiding permutations  are there that are composed of only 3-cycles, 2-cycles, and fixed points?
\end{question}

It seems like since there can be at most two 3-cycles in a 123-avoiding permutation, and the answer for 2-cycles and fixed points is very nice, it should be a straightforward exercise to answer this question. However, most proofs that 123-avoiding involutions are enumerated by $\binom{n}{\lfloor \frac{n}{2}\rfloor}$ use the Robinson–Schensted–Knuth algorithm and standard Young tableaux, or symmetric Dyck paths, neither of which interact well with 3-cycles. 


Finally, we could also consider the pattern 321. In \cite{AG21}, the authors find a formula for the number of permutations avoiding 321 composed only of 3-cycles, but the answer is a sum over Dyck paths. It would be nice to find a closed form or generating function for these permutations, as well as order three 321-avoiding permutations. 

\begin{question}
    Is there a closed form for the number of 321-avoiding permutations composed only of 3-cycles? How many 321-avoiding permutations are there of order 3? How many  are there that are composed of only 3-cycles, 2-cycles, and fixed points? 
\end{question}

This question seems especially difficult, so we include this conjecture regarding the relationship of 321-avoiding permutations to other pattern-avoiding permutations. See Tables~\ref{tab:order3} and \ref{tab:1-2-3}  for the numerical evidence for this conjecture. 

\begin{conjecture}
    For $n\geq 1,$ we have
    \[ a_n^{\{1,3\}}(132) \leq a_n^{\{1,3\}}(321)\]
    and 
    \[a_n^{\{1,2,3\}}(123) \leq a_n^{\{1,2,3\}}(132)\leq a_n^{\{1,2,3\}}(321)\leq a_n^{\{1,2,3\}}(231).\]
\end{conjecture}


\begin{table}[ht]
\centering\makegapedcells
\renewcommand{\arraystretch}{1.6}
\begin{tabular}{|c|c|c|c|c|c|c|c|c|c|c|c|c|}
\hline
$\sigma$ & 1& 2 & 3 &4 & 5 & 6 & 7 & 8 & 9 &10 & 11 &12 \\ \hline\hline
123 & 1 & 1 & 2 & 4 & 2 & 6 & 12 & 6 & 0 & 0 & 0 &0 \\ \hline
132 & 1 & 1 & 3 & 5 & 7 & 17 & 31 & 49 & 107 & 201 & 339 & 699 \\ \hline
231 & 1 & 1 & 2 & 5 & 9 & 16 & 32 & 61 & 114 & 219  & 418  & 792 \\ \hline
321 & 1 & 1 & 3 & 5 & 7 & 19 & 35 & 55 & 139 & 267 & 447 & 1075 \\ \hline
\end{tabular}
\caption{The number of order 3 permutations (i.e., those composed only of 3-cycles and fixed points) that avoid $\sigma$ for $n\in[12].$}
\label{tab:order3}
\end{table}

\begin{table}[ht]
\centering\makegapedcells
\renewcommand{\arraystretch}{1.6}
\begin{tabular}{|c|c|c|c|c|c|c|c|c|c|c|c|c|}
\hline
$\sigma$ & 1& 2 & 3 &4 & 5 & 6 & 7 & 8 & 9 &10 & 11 &12 \\ \hline\hline
123 & 1 & 2 & 5 & 10 & 22 & 50 & 103 & 222 &432 & 898 & 1700 & 3482 \\ \hline
132 & 1 & 2 & 5 & 10 & 24 & 56 & 127 & 300 & 698 & 1620 & 3794 & 8848 \\ \hline
231 & 1 & 2 & 5 & 12  & 29 & 71 & 171 & 411 & 990 & 2380 & 5722 & 13765 \\ \hline
321 & 1 & 2 & 5 & 10 & 24 & 58 & 133 & 324 & 782 & 1868 & 4550 & 11030 \\ \hline
\end{tabular}
\caption{The number of $\sigma$-avoiding permutations composed only of 3-cycles, 2-cycles, and fixed points for $n\in[12].$}
\label{tab:1-2-3}
\end{table}

\end{document}